\documentclass{ws-jktr}

\usepackage{lipsum} 
\usepackage{picins} 
\usepackage{graphicx} 
\usepackage{color}
\usepackage{mathrsfs}  
\usepackage{picinpar}
\begin{document}

\markboth{Deng, Jin, Kauffman}
{Graphical virtual links and a polynomial of signed cyclic graphs}
\catchline{}{}{}{}{}
\title{Graphical virtual links and a polynomial of signed cyclic graphs}

\author{Qingying Deng$^{\dag}$, Xian'an Jin$^{\dag}$\thanks{Corresponding author}, Louis H. Kauffman$^{\ddag}$}

\address{$^{\dag}$School of Mathematical Sciences,\\ Xiamen University, Xiamen, Fujian 361005, \\P. R. China \\
dqygzl@sina.com(Qingying Deng)}

\address{$^{\dag}$School of Mathematical Sciences,\\ Xiamen University, Xiamen, Fujian 361005, \\P. R. China \\
xajin@xmu.edu.cn(Xian'an Jin)}

\address{$^{\ddag}$Department of Mathematics, Statistics, and Computer Science,\\ University of Illinois at Chicago,  Chicago\\USA \\
loukau@gmail.com(Louis H. Kauffman)}

\maketitle

\begin{abstract}
For a signed cyclic graph $G$, we can construct a unique virtual link  $L$ by taking the medial construction and convert 4-valent vertices of the medial graph to crossings according to the signs.
If a virtual link can occur in this way then we say that the virtual link is $\textit{graphical}$.
In the article we shall prove that a virtual link $L$ is graphical if and only if it is checkerboard colorable.
On the other hand, we introduce a polynomial $F[G]$ for signed cyclic graphs, which is defined via a deletion-marking recursion.
We shall establish the relationship between $F[G]$ of a signed cyclic graph $G$ and the bracket polynomial of one of the virtual link diagrams associated with $G$.
Finally we give a spanning subgraph expansion for $F[G]$.
\end{abstract}

\keywords{virtual link; signed cyclic graph; orientable ribbon graphs; graphical; checkerboard colorable; $F[G]$ polynomial; bracket polynomial.}

\section{Introduction}
Virtual knot theory, introduced by Kauffman in \cite{Kau1}, is an extension of classical knot theory.
In knot theory, it is well known that signed plane graphs are, by a medial construction, in one-to-one correspondence with
diagrams for classical knots and links. Here a plane graph means a planar graph $G$ together with an embedding of this graph in the plane.
A $\textit{cyclic graph}$ consists of an abstract graph $G$ with a cyclic ordering of the half-edges at each vertex, or, equivalently, it is a cellularly embedding of a graph into a closed orientable surface.
The theory of cyclic graphs is equivalent to the theory of graphs on oriented surfaces.
Note that here we ignore the way a cyclic graph is embedded into $\mathbb{R}^3$.
However, if it is impossible to embed a cyclic graph in plane, we
may immerse it by marking artifacts of the immersion (we assume the immersion to be
generic) by small circles (virtual crossings).
Note that there are a lot of immersions, but we just need to choose one.
Without loss of generality, we consider that the cyclic ordering of each vertex is counterclockwise.
A $\textit{signed graph}$ is a graph whose edges are labelled with + or -. A $\textit{signed cyclic graph}$ is a cyclic graph with edge signs.
For a 2-cellularly embedded graph into a closed orientable surface, we can construct a link diagram on the surface via the medial construction, which represents a unique virtual link.
{There are many ways of immersing a signed cyclic graph into the plane.
For each such immersion, we can obtain a virtual link diagram as shown in Fig. \ref{vi}.
These are called the virtual link diagrams $\textit{associated with}$ a signed cyclic graph.}
They are all equivalent (i.e. representing the same virtual link) up to generalised Reidemeister moves (See Fig. \ref{RM}).
If a virtual link occurs in this way then we say that the virtual link is $\textit{graphical}$.
Conversely, for a checkerboard colorable link diagram, we adopt a natural method, which is first introduced by Chmutov and Pak \cite{Chmu07}, to construct an associated $\textit{signed ribbon graph}$ in Section \ref{GCB}.
We do not go into detail about ribbon graphs, but refer the interested reader to \cite{Mo,Chmu09}.

Kauffman \cite{Kau2} posed the following problem: does there exist a graphical non-trivial virtual knot with unit Jones polynomial?
In this article we answer the question, and this further motivates us to characterize the family of graphical virtual links.
In section 3, we show how to translate the virtualization operation on link diagrams to an operation on signed cyclic graphs.
This gives us access to infinitely many signed cyclic graphs whose associated virtual link diagrams have trivial Jones polynomials.
{A checkerboard colorable virtual link $L$ has a checkerboard colorable diagram $D$ (see Section 2.3).
A signed Tait (cyclic) graph $G$ can be constructed from the diagram $D$. In turn, via medial construction, we can obtain a diagram $D'$ associated with $G$, which is equivalent to the diagram $D$.
We shall prove that a virtual link $L$ is graphical if and only if it is checkerboard colorable in Section \ref{GCB}.}

Another motivation for this paper is to extend an analog of the relationship between signed plane graphs and classical link diagrams to signed cyclic graphs and virtual link diagrams (see Section \ref{poly}).
On the one hand, the bracket polynomial of a classical link diagram equals the signed Tutte polynomial of its signed plane graph \cite{Kau4}.
On the other hand, for a checkerboard colorable virtual link diagram $D$, Chmutov and Pak \cite{Chmu07} construct a corresponding signed ribbon graph $\widehat{G}$ from it and discuss the relationship between the bracket polynomial of $D$ and the signed Bollob\'{a}s-Riordan polynomial of $\widehat{G}$.
Here we construct a polynomial $F[G]$ for a signed cyclic graph via a deletion-marking recurrence and establish the relationship between $F[G]$ and bracket polynomial of $D_G$ which is one of the virtual link diagrams associated with $G$ by medial construction in Section \ref{poly}.
Finally we give a spanning subgraph expansion for $F[G]$.
In fact, $F[G]$ is a specialization of the signed Bollob\'{a}s-Riordan polynomial.

\begin{figure}[!htbp]
  \centering
  \includegraphics[width=4.5in]{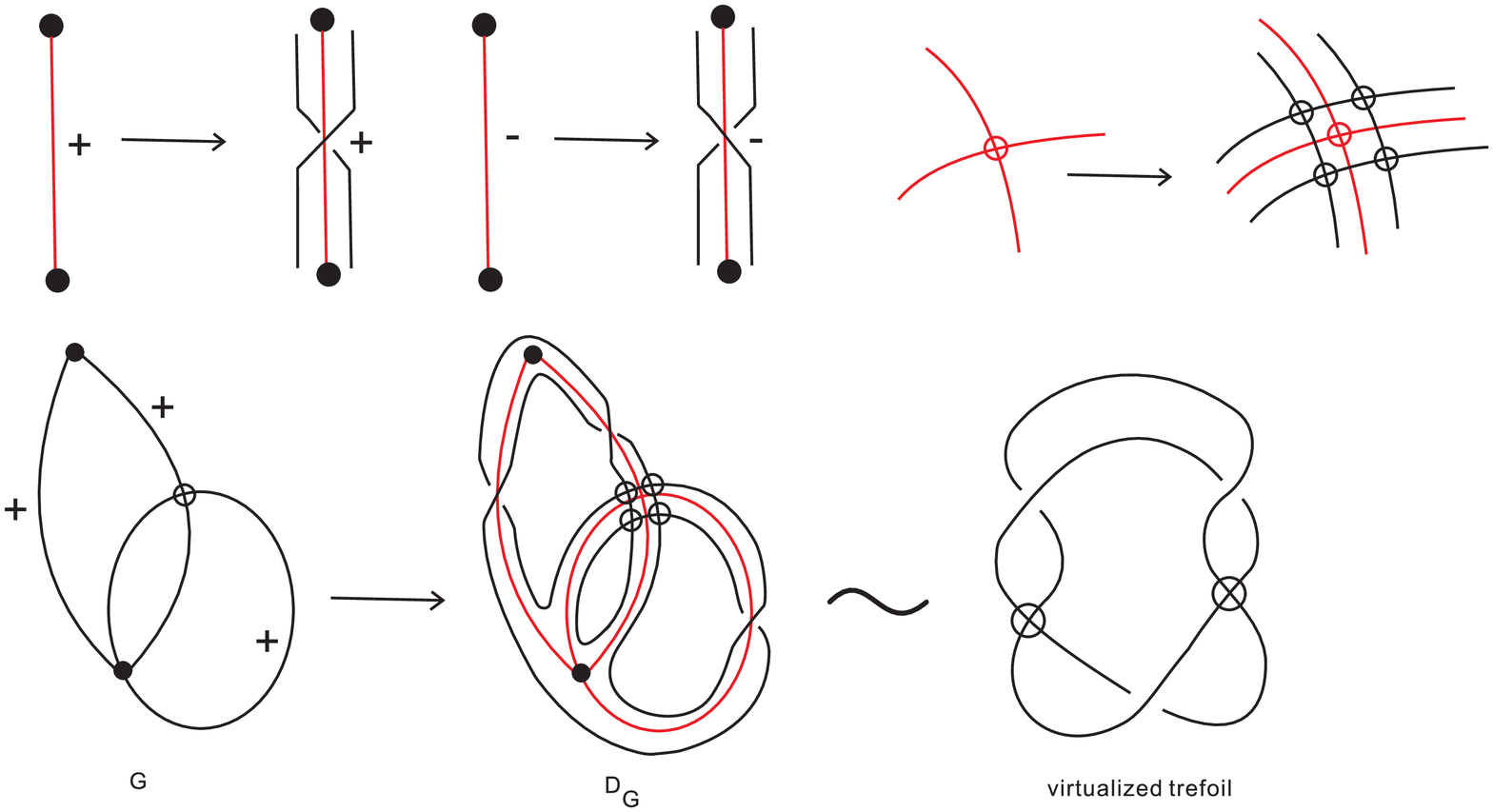}
  \renewcommand{\figurename}{Fig.}
\caption{{\footnotesize Associate a signed cyclic graph with a unique virtual link.}}\label{vi}
\end{figure}

\section{Preliminaries}
\subsection{Ribbon graphs}\label{ribbon}
A $\textit{ribbon graph}$ $\textbf{G}=(V(\textbf{G}),E(\textbf{G}))$ is a surface with boundary, represented as the union of two sets of discs: a set $V(\textbf{G})$ of vertices and a set $E(\textbf{G})$ of edges such that:

(1) the vertices and edges intersect in disjoint line segments;

(2) each such line segment lies on the boundary of precisely one vertex and precisely one edge; and

(3) every edge contains exactly two such line segments.

See Fig. \ref{ir} for an example of ribbon graphs.

It is well-known that ribbon graphs are equivalent to cellularly embedded graphs.
Ribbon graphs arise naturally from small neighborhoods of cellularly embedded graphs.
On the other hand, topologically, a ribbon graph is a surface with boundary.
Capping-off the holes results in a band decomposition, which gives rise to a cellularly embedded graph in the obvious way.
A ribbon graph is $\textit{orientable}$ if it is orientable when viewed as a punctured surface.
Two ribbon graphs are $\textit{equivalent}$ if there is a homeomorphism taking one to the other that preserves the vertex-edge structure.
The homeomorphism should be orientation preserving when the ribbon graphs are orientable.
The homeomorphism should be sign-preserving when signed ribbon graphs are discussed.
We consider ribbon graphs up to this equivalence.

The concept of a (signed) cyclic graph is $\textit{equivalent}$ to the concept of an orientable (signed) ribbon graph.
Two signed cyclic graphs are $\textit{equivalent}$ if and only if their corresponding orientable ribbon graphs are equivalent.
In this paper, we will use the two equivalent representations interchangeably, using whichever best facilitates the discussion at hand.
When the two equivalent representations appear at the same time, sometimes we let $\textbf{G}_r$ denote the signed ribbon graph associated with the signed cyclic graph $G$.
We refer interested readers to \cite{Boll,Boll2,Chmu17} and its references.

\begin{figure}[!htbp]
  \centering
  \includegraphics[width=2.5in]{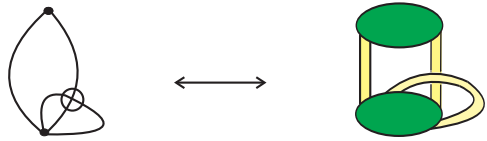}
  \renewcommand{\figurename}{Fig.}
\caption{{\footnotesize The cyclic graph $G$ and its associated orientable ribbon graph $\textbf{G}_r$.}}\label{ir}
\end{figure}

\subsection{Virtual link diagrams}
A $\textit{virtual link diagram}$ is a closed 1-manifold generically immersed in $\mathbb{R}^2$ such that each double point is labeled to be (1) a real crossing which is indicated as usual in classical knot theory or (2) a virtual crossing which is indicated by a small circle around the double point.
The moves for virtual link diagrams illustrated in Fig. \ref{RM} are called $\textit{generalized Reidemeister moves}$. The $\textit{detour move}$ (see Fig. \ref{dm}) is a consequence of $(B)$ and $(C)$ in Fig. \ref{RM}.
Two virtual link diagrams are said to be $\textit{equivalent}$ if they are related by a finite sequence of generalized Reidemeister moves.
We call the equivalence class of a virtual link diagram a $\textit{virtual link}$.
We refer the reader to \cite{Kau1,Kau3} for further details about virtual links.
\begin{figure}[pb]
\centering
\includegraphics[width=3.5in]{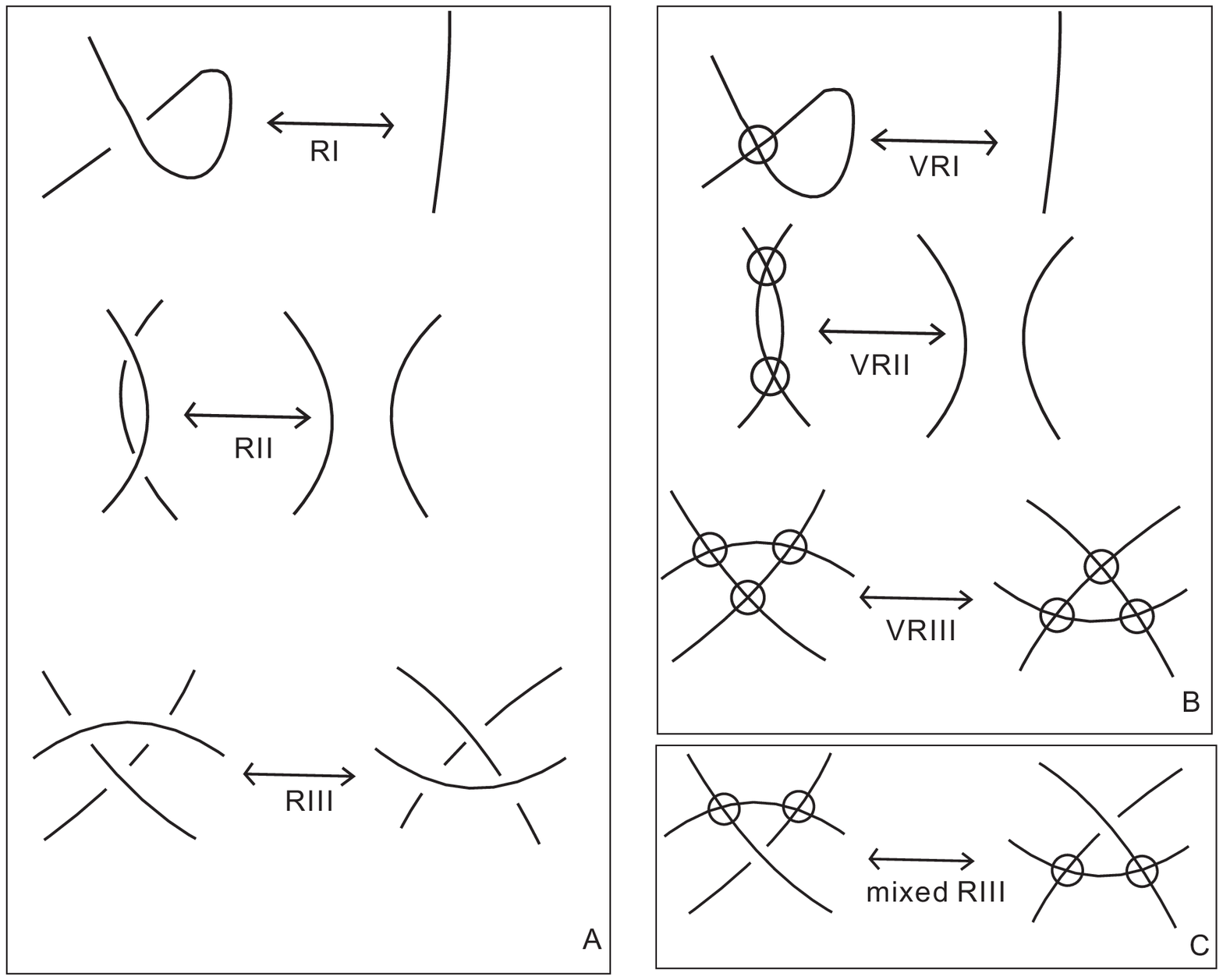}
\renewcommand{\figurename}{Fig.}
\caption{{\footnotesize Generalized Reidemeister Moves.}}\label{RM}
\end{figure}
\begin{figure}[pb]
\centering
\includegraphics[width=3.5in]{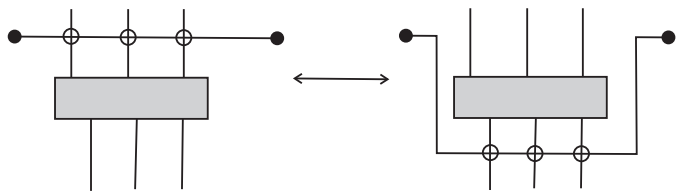}
\renewcommand{\figurename}{Fig.}
\caption{{\footnotesize Detour Move.}}\label{dm}
\end{figure}


\subsection{Checkerboard colorable virtual links}
The notion of a checkerboard coloring of a virtual link diagram was introduced by Kamada in \cite{K1,K2}.
A virtual link diagram is said to be $\textit{checkerboard colorable}$ if there is a coloring of a small neighbourhood of one side of each arc in the diagram such that near a classical crossing the coloring alternates, and near a virtual crossing the colorings go through independent of the crossing strand and its coloring.
A virtual link is said to be $\textit{checkerboard colorable}$ if it has a checkerboard colorable diagram.
Two examples are given in Fig. \ref{ckb}.

\begin{figure}[!htbp]
  \centering
  \includegraphics[width=3.5in]{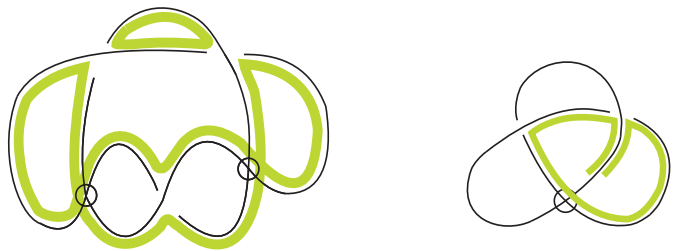}
  \renewcommand{\figurename}{Fig.}
\caption{{\footnotesize The left virtual link diagram is checkerboard colorable, the right is not checkerboard colorable.}}\label{ckb}
\end{figure}

Note that every checkerboard colorable virtual link diagram has two kinds of checkerboard colorings and not every virtual link diagram is checkerboard colorable.
Checkerboard colorability of a virtual link diagram is not necessarily preserved by generalized Reidemeister moves.
There are five {ways} to detect checkerboard colorability {discussed} in \cite{Ima}. But there are still some virtual knots whose checkerboard colorability has not been detected.
{For example, it is, at this writing, unknown whether the Kishino virtual knot (4.55 in \cite{Gr}) \cite{Ima,Kau3} is checkerboard colorable. }

\section{Graphical virtual links}\label{GCB}
\subsection{Examples}
Given a crossing $i$ in a link diagram, we define the $\textit{virtualization}$ $v(i)$  of the crossing by the local replacement indicated in Fig. \ref{vi2}, i.e., the original crossing is replaced by a crossing that is flanked by two virtual crossings.
It is known that the virtualized trefoil (see Fig. \ref{vi}) is non-trivial and non-classical virtual diagram with unit Jones polynomial \cite{Kau1}.
We found that the virtualized trefoil is graphical as shown in Fig. \ref{vi}, giving a positive answer to {Kauffman's problem}.
On the other hand, we shall explain its graphicality since changing an edge (i.e. take the partial dual of an edge in \cite{Mo}) of a signed cyclic graph $G$ corresponds to virtualizing a real crossing of link diagram $D_G$ associated with $G$ as illustrated in Fig. \ref{fig12}.
{The construction can be generalized to any classical knot diagram by selecting a subset of crossings whose switchings unknot the diagram.
By virtualizing these crossings while leaving the rest of the diagram just as before, we can obtain a graphical virtual knot diagram with unit Jones polynomial which is non-trivial and non-classical \cite{DKK}.
We refer the reader to examine \cite{DKK,Kau1,Kau3,FIKM} for the details of this construction.
}

\begin{figure}[!htbp]
\centering
\includegraphics[width=3.5in]{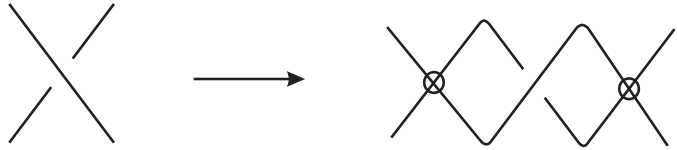}
\renewcommand{\figurename}{Fig.}
\caption{{\footnotesize Virtualize a real crossing.}}\label{vi2}
\end{figure}
\begin{figure}[!htbp]
\centering
\includegraphics[width=3.5in]{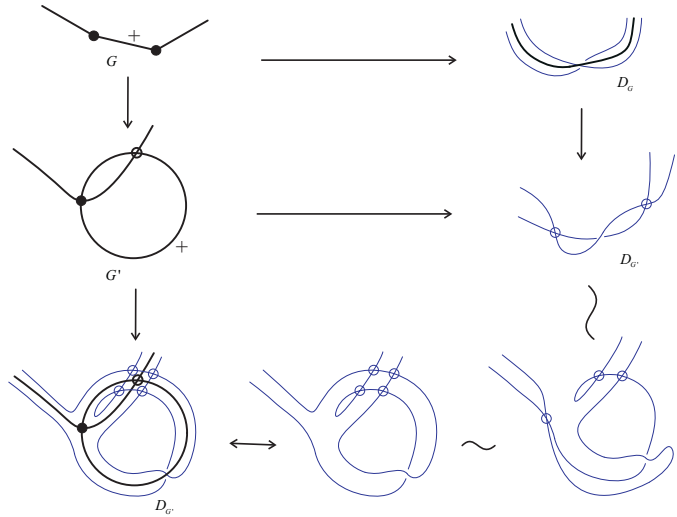}
\renewcommand{\figurename}{Fig.}
\caption{{\footnotesize Change an edge of $G$ and virtualize a real crossing of $D_G$.}}\label{fig12}
\end{figure}

\subsection{A necessary and sufficient condition of graphicality}\label{cond}
\begin{lemma}\label{main1} Let $G$ be a signed cyclic graph and {let $D_G$ be a
virtual link diagram} associated with $G$. Then $D_G$ is checkerboard colorable.
\end{lemma}
\begin{proof}
For each real crossing of $D_G$, at the positive (resp. negative) edge of $G$, we shall take $B-$smoothing (resp. $A-$smoothing), then $D_G$ becomes a set of $|V(G)|$ disjoint closed curves.
{We shade the closed curves using a collection of shaded annuli
as shown in Fig. \ref{ab} (1) and (2).
Then the shading will induce a checkerboard coloring of $D_G$, completing the proof.}
\end{proof}

\begin{figure}[!htbp]
  \centering
  \includegraphics[width=4.5in]{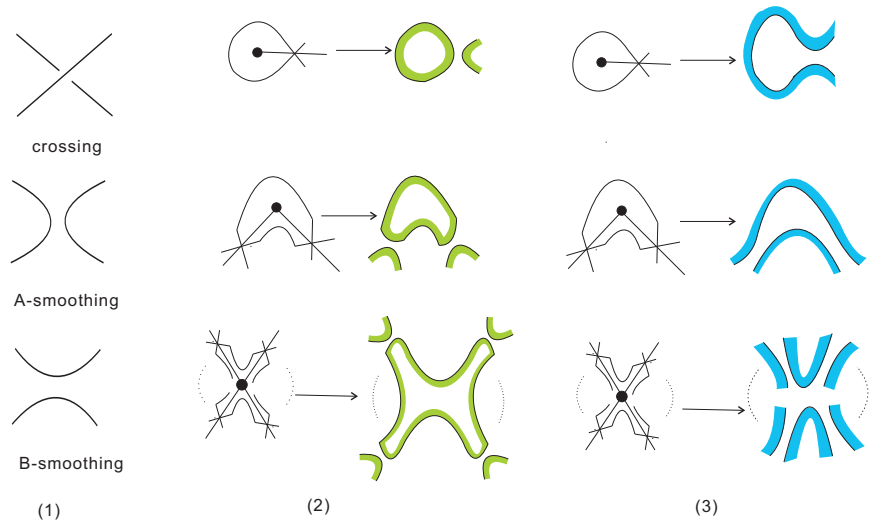}
  \renewcommand{\figurename}{Fig.}
\caption{{\footnotesize Checkerboard colorings of virtual link diagrams associated with cyclic graphs.}}\label{ab}
\end{figure}

\begin{remark}
In the above proof, using similar method, we take the $A-$smoothing (resp. $B-$smoothing) for each real crossing of $D_G$ at the positive (resp. negative) edge of $G$.
Then $D_G$ becomes a set of $|bc(G)|$ disjoint closed curves, which correspond to the boundary components of $G$.
{We shade the closed curves using a collection of shaded annuli as shown in Fig. \ref{ab} (1) and (3).}
Then the shading will induce another checkerboard coloring of $D_G$.
\end{remark}

Let $D$ be a checkerboard colorable virtual link diagram and $C$ be a checkerboard coloring of $D$.
Topologically, a coloring is represented by a connection of annuli.
Each annulus has two boundary circles, an exterior circle which goes along the link except small arcs near real crossings where it jumps from one strand to another one, and an interior circle.
On the one hand, in order to construct a ribbon graph from a checkerboard colorable virtual link diagram we replace every crossing by a signed edge-ribbon connecting the corresponding arcs of the exterior circles as shown in Fig. \ref{vb}.
The signed ribbon graph $\textbf{G}_D$ is obtained by gluing discs along the interior circles of the annuli of the coloring.
The approach is first introduced by Chmutov and Pak in \cite{Chmu07}.
An example is given in Fig. \ref{GH}.

\begin{figure}[!htbp]
  \centering
  \includegraphics[width=4.5in]{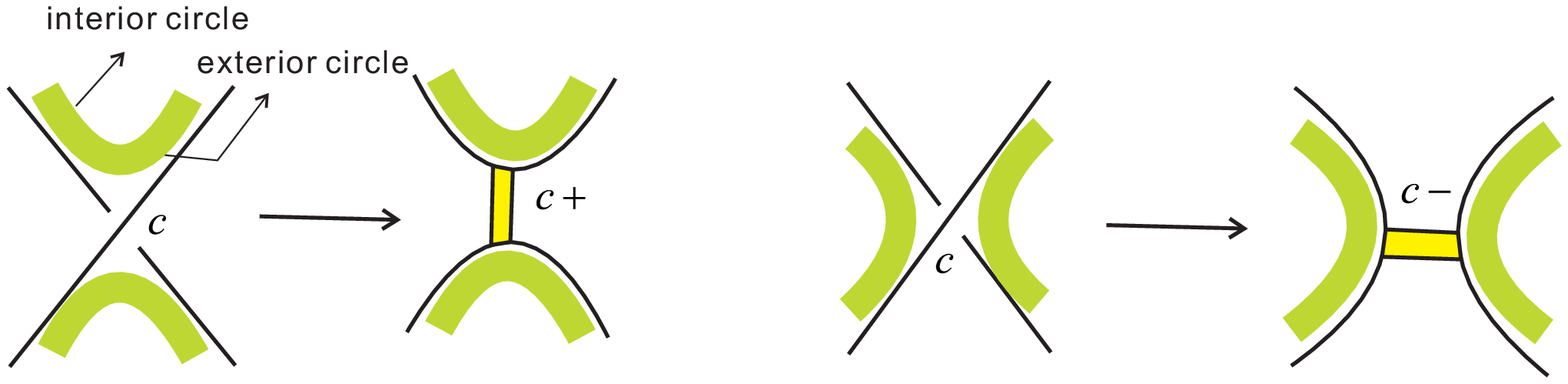}
  \renewcommand{\figurename}{Fig.}
\caption{{\footnotesize The real crossings of $D$ corresponds to the edge-ribbons of ribbon graph.}}\label{vb}
\end{figure}

\begin{figure}[!htbp]
  \centering
  \includegraphics[width=4.0in]{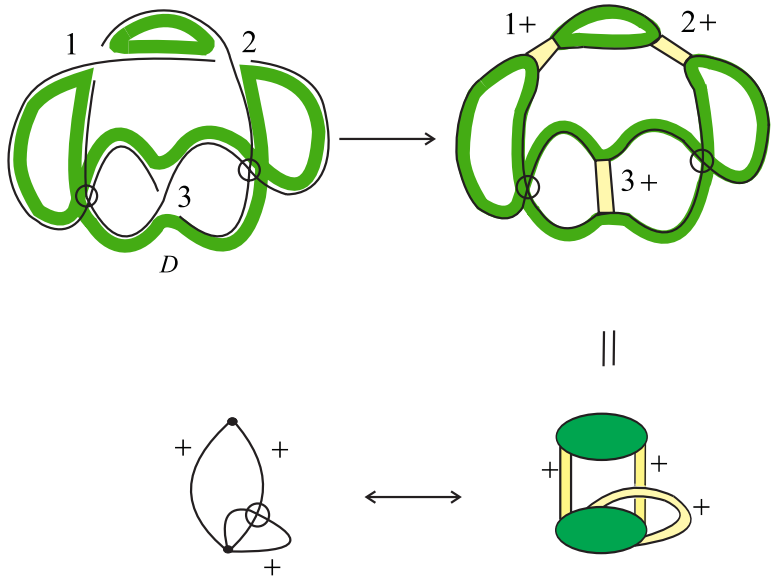}
  \renewcommand{\figurename}{Fig.}
\caption{{\footnotesize An example of constructing $\textbf{G}_D$ from checkerboard coloring of $D$.}}\label{GH}
\end{figure}

Let $U$ be the corresponding 4-valent cyclic graph of $D$ after regarding each real crossing of $D$ as a 4-valent vertex.
{Let $\textbf{U}_r$ be the signed ribbon graph associated with the signed cyclic graph $U$.}
Note that classical crossings correspond to disks with attached bands and virtual crossings becomes non-touching bands as shown in Fig. \ref{rv}.
The ribbon graph $\textbf{U}_r$ obtained has boundary curves.
Each boundary curve is capped with a disk.

Note that for another checkerboard coloring of $D$, we can similarly obtain a signed ribbon graph $\textbf{G}^*_D$, which is the geometric dual of $\textbf{G}_D$.
{Then $D$, $\textbf{G}_D$ and $\textbf{G}^*_D$ are both embedded into the closed surface formed by capping off the holes of $\textbf{U}_r$.}
See Fig. \ref{rv}.
\begin{figure}[!htbp]
  \centering
  \includegraphics[width=3.5in]{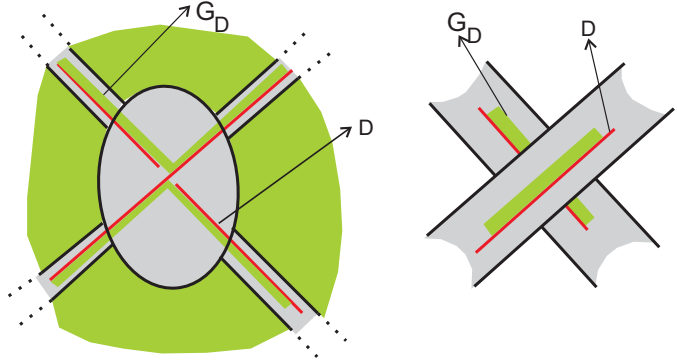}
  \renewcommand{\figurename}{Fig.}
\caption{{\footnotesize {Real crossings of $D$ are embedded in $\textbf{U}_r$. Virtual crossings of $D$ become bands in $\textbf{U}_r$.
$\textbf{G}_D$ is also embedded into the closed surface obtained from $\textbf{U}_r$ by capping off holes with disks.}}}\label{rv}
\end{figure}

\begin{lemma}\label{lm2}
Let $D$ be a checkerboard colorable virtual link diagram,
$\textbf{G}_D$ and $\textbf{G}^*_D$ be the corresponding signed ribbon graphs $D$. Then they are orientable.
\end{lemma}
\begin{proof}
This is because $\textbf{U}_r$ is orientable.
\end{proof}

\begin{theorem}\label{main2}
Let $L$ be a  virtual link, then $L$ is graphical if and only if it is checkerboard colorable.
\end{theorem}
\begin{proof}
The necessity of Theorem \ref{main2} follows from Lemma \ref{main1}.
Now we prove the sufficiency.
Let $D$ be a checkerboard colorable link diagram of $L$, and $\textbf{G}_D$ be one of the corresponding signed ribbon graphs of $D$.
Let $G$ be the core graph of $\textbf{G}_D$.
It is sufficient to explain that we can obtain virtual link diagram $D$ from $G$ via medial construction.
Let $U$  be the underlying 4-valent cyclic graph of $D$ and
{let $F=\textbf{U}_r$ be the signed ribbon graph associated with the signed cyclic graph $U$.}
Then $D$ can be viewed as a link diagram on $F$ and a link in $F\times I$, and $G$ can be embedded in $F$.
Then $D$ can be transformed into the diagram $D'$ constructed from $G$ via the medial construction on $F$ by isotopies of $F$ as shown in Fig. \ref{all}.
Note that $P=(\textbf{G}_D,D')$ is an abstract link diagram \cite{Kam} which can be immersed into the plane and accordingly $G$ is also immersed into the plane. Then $L$ is graphical from $G$.

\begin{figure}[!htbp]
  \centering
  \includegraphics[width=5in]{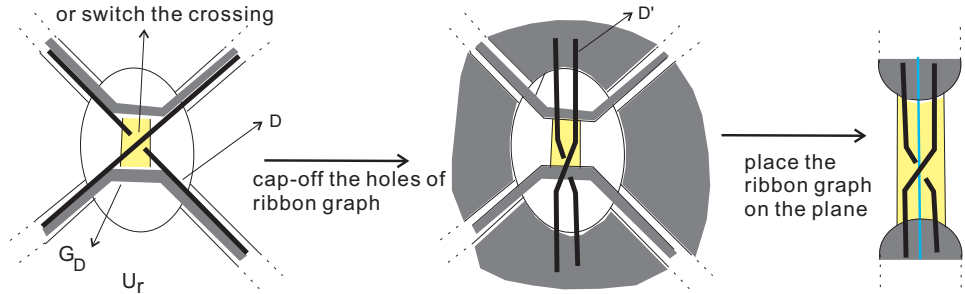}
  \renewcommand{\figurename}{Fig.}
\caption{{\footnotesize (1) In the left figure, {the bold line} represents $D$ which is embedded into ribbon graph $\textbf{U}_r$, {the shaded annulus} represents a checkerboard coloring of $D$, from where $\textbf{G}_D$ are constructed. (2) $\textbf{G}_D$ and deformation of $D$ to $D'$. (3) In the right figure, we immerse $P$ into the plane.}}\label{all}
\end{figure}
\end{proof}

\begin{remark}
We call $\textbf{G}_D$ and $\textbf{G}^*_D$ (or $G_D$ and $G^*_D$) {the $\textit{signed Tait graphs}$} of a checkerboard colorable virtual link diagram $D$.
Moveover, in \cite{Mo}, a set of ribbon graphs are constructed from $D$, which are partial dual to each other and includes the two signed Tait graphs above.
\end{remark}

\begin{corollary}\label{mcor}
Let $L$ be a checkerboard colorable virtual link, $D$ be a checkerboard colorable diagram of $L$ and $G_D$ be one of the signed Tait graphs of $D$. Then $L$ is graphical from $G_D$.
\end{corollary}

\section{A polynomial for signed cyclic graphs}\label{poly}
For a virtual link diagram $D$, the $\textit{bracket polynomial}$ $\langle D\rangle=\langle D\rangle(A,B,d)$ \cite{Kau1,Kau3} can be defined recursively by using the following three relations:
\begin{equation*}
  {\centering \includegraphics[width=2.0in]{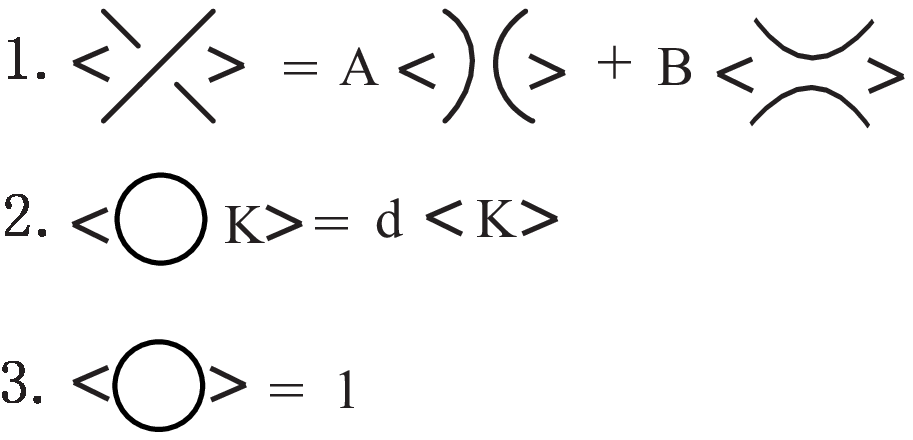}}
\end{equation*}
\noindent Here it is understood that the three small diagrams are parts of otherwise identical larger diagrams and $\bigcirc$ denotes a diagram with no real crossings.

There also is a state-space expansion for bracket polynomial as follows.

\begin{definition}
The bracket polynomial of a virtual link diagram $D$ is a polynomial in three variables $A,B,d$ defined by the formula

$$\langle D\rangle(A,B,d)=\sum_{\sigma\in \mathscr{S}(D)}A^{\alpha(\sigma)}B^{\beta(\sigma)}d^{|\sigma|-1},$$

\noindent where $\alpha(\sigma)$ and $\beta(\sigma)$ are the numbers of $A$-smoothings and $B$-smoothings (see Fig. \ref{ab}) in a state $\sigma$, respectively, $\mathscr{S}(D)$ denotes the set of states of $D$, $|\sigma|$ is the number of closed curves in $\sigma$.
\end{definition}

In particular, it's well-known that if $B=A^{-1}$, $d=-(A^2+A^{-2})$, then $\langle D\rangle$ is invariant under the Reidemeister moves of type II and type III and generalized Reidemeister moves.
The bracket polynomial is normalized to an invariant $f_K(A)$ of all the moves by the formula $f_K(A)=(-A^3)^{-w(D)}\langle D\rangle$ where $w(D)$ is the writhe of the oriented diagram $D$.
The writhe is the sum of the orientation signs ($\pm1$) of the crossings of the diagram.
Here the sign is the $\textit{oriented sign}$ shown in Fig. \ref{sign}.
The $\textit{Jones polynomial}$ of an oriented virtual link $K$, $V_K(t)$ is given in terms of this model by the formula:

$$V_K(t)=f_K(t^{-\frac{1}{4}}).$$

We refer interested readers to \cite{Kau1,Kau3}.

\begin{figure}[!htbp]
  \centering
  \includegraphics[width=2.0in]{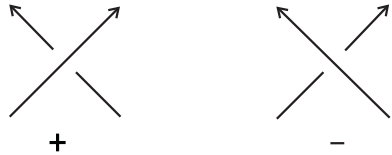}
  \renewcommand{\figurename}{Fig.}
  \caption{{\footnotesize The oriented signs of a link diagram.}}\label{sign}
\end{figure}

For a real crossing \includegraphics[width=1.5cm,keepaspectratio]{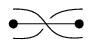}, there are two ways to resolve a real crossing, that is, \includegraphics[width=1.5cm,keepaspectratio]{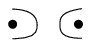} and \includegraphics[width=1.5cm,keepaspectratio]{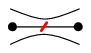} (the edge is marked with a red skew line as Definition \ref{def2}). Hence we can define a polynomial for a signed cyclic graph, which has closely relationship with bracket polynomial for a virtual link diagram.

A $\textit{spanning subgraph}$ $H$ of a cyclic graph $G$ is a subgraph consisting of all vertices of $G$ and a subset of the edges of $G$ which respects the cyclic order of $G$.
When every edge of $H$ is a marked edge, we call $H$ a $\textit{marked spanning subgraph}$ of $G$. Specially, $H$ can be a empty graph which only has vertices but no edges. Let $bc(H)$ be the number of boundary components for ribbon graph $\textbf{H}_r$.

\begin{definition}\label{def2}
Let $G$ be a signed cyclic graph, then the defining formulas for $F[G]$ are:

(1)For any $e\in E(G)$, let $G(\overline{e})$ be the graph obtained from $G$ by keeping $e$ and marking $e$ with a red skew line as a special edge, then

$F[G]=B\cdot F[G-e]+A\cdot F[G(\overline{e})]$ if sign($e$)$>$0,

$F[G]=A\cdot F[G-e]+B\cdot F[G(\overline{e})]$ if sign($e$)$<$0.

(2) If $H$ is a marked spanning subgraph of $G$, then $F[H]=d^{bc(H)-1}$.

\end{definition}
\ \ \ \ \ \ \ \ \ \ \ \ \ \ \ \ \ \ \includegraphics[width=3.5in,keepaspectratio]{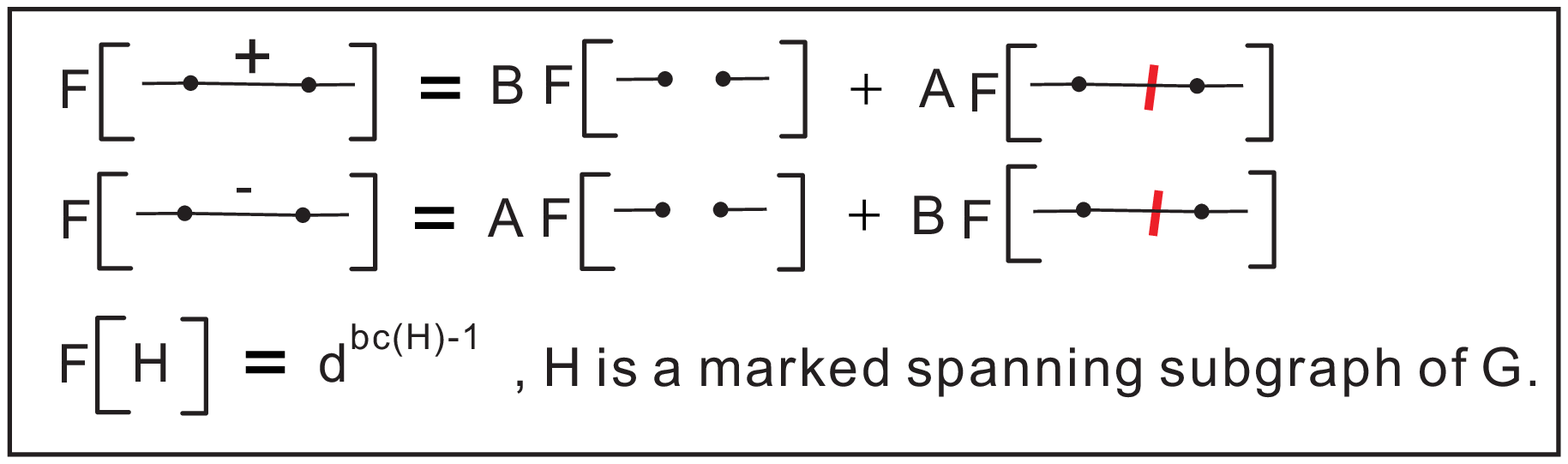}

$\textbf{Example 4.1:}$
Calculate $F[G]$ for a signed cyclic graph \includegraphics[width=1.0cm,keepaspectratio]{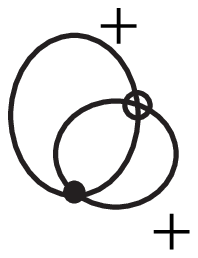}.

\ \ \ \ \ \ \ \ \ \ \ \ \  \ \includegraphics[width=3.5in,keepaspectratio]{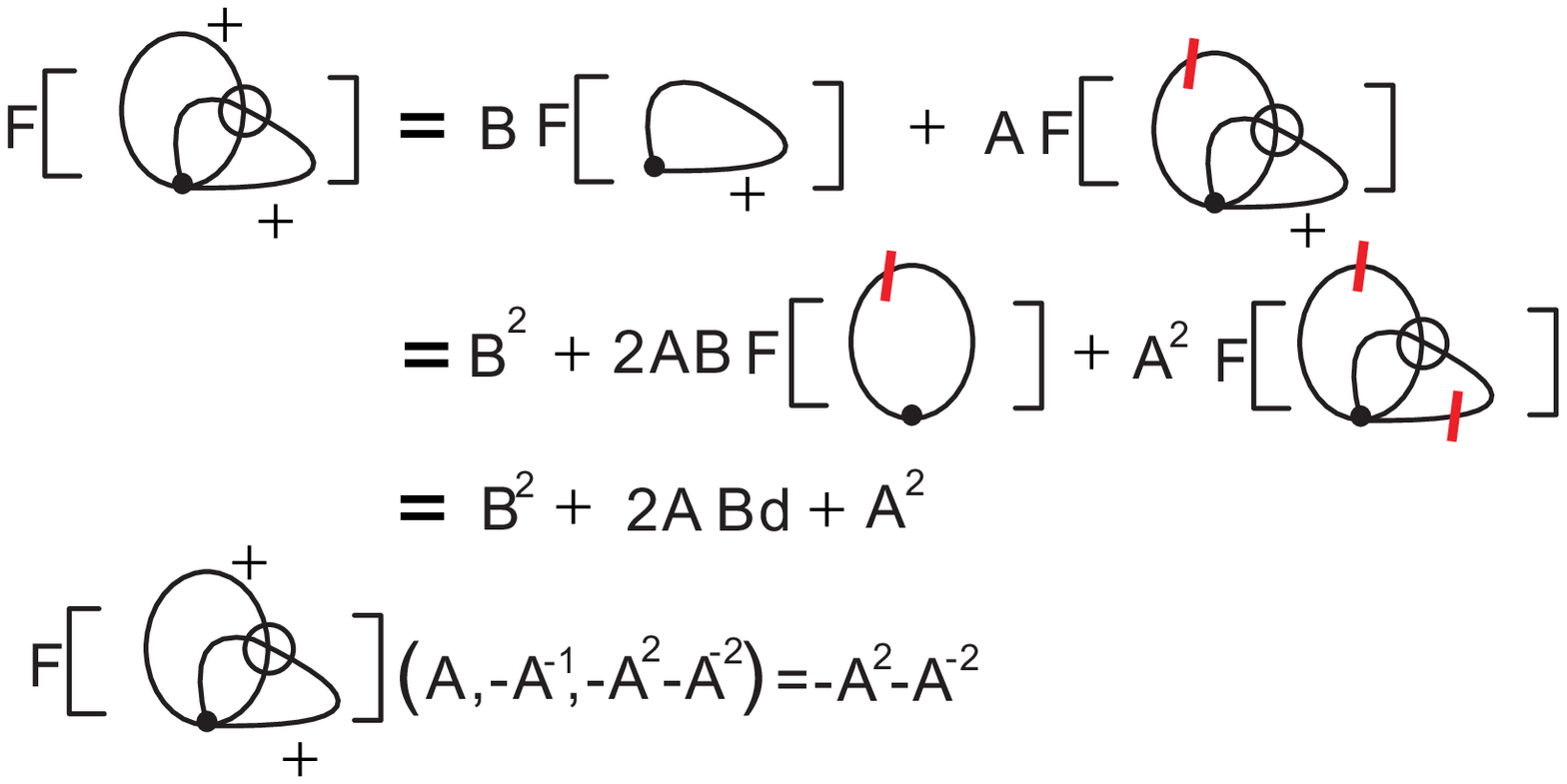}

\begin{proposition}\label{GD}
Let $G$ be a signed cyclic graph. Let $D_G$ be one of the virtual link diagrams associated with $G$ via the medial construction. Then
$$F[G](A,B,d)=\langle D_G\rangle(A,B,d).$$
\end{proposition}

\begin{proof}
The proof follows immediately from the medial construction.
The formulas for the bracket polynomial directly translates to the recursion formulas (1)(2) for $F[G]$.
\end{proof}

Based on the above Proposition \ref{GD}, we explain that $F[G]$ is well-defined.
\\
\\
\indent Let $\mathscr{S}(G)$ denote the set of spanning subgraphs of $G$. Let $e_{\pm}(S)$ denotes the number of positive or negative edges in $S$. Denote $G-S$ the spanning subgraph of $G$ with exactly those edges of $G$ that do not belong to $S$.

Similar to bracket polynomial, for $F[G]$ we can obtain a spanning subgraph expansion:
\noindent
\begin{equation}
F[G](A,B,d)=\sum_{S\in \mathscr{S}(G)}A^{e_{+}(S)+e_{-}(G-S)}B^{e_{-}(S)+e_{+}(G-S)}d^{bc(S)-1}
\label{this}
\end{equation}

\noindent where the sum is over all (marked) spanning subgraph $S$ of $G$.

\begin{proposition}\label{lg}
Let $D$ be a checkerboard colorable virtual link diagram and $G_D$ be a signed Tait graph of $D$. Then
$$\langle D\rangle(A,B,d)=F[G_D](A,B,d).$$
\end{proposition}
\begin{proof}
It follows from the proof process of Theorem \ref{main2} and Proposition \ref{GD}.
\end{proof}

\begin{corollary}\label{Jones}
Let $L$ be an oriented graphical virtual link. Let $D$ be a graphical diagram of $L$.
Let $G_D$ be a signed Tait graph of $D$.
Then $$V_L(t)=(-t^{-\frac{3}{4}})^{-w(D)}F[G_D](t^{-\frac{1}{4}},t^{\frac{1}{4}},-t^{-\frac{1}{2}}-t^{\frac{1}{2}}).$$
\end{corollary}

\begin{remark}
We can show that $F[G]$ is the specialization of signed Bollob\'{a}s-Riordan polynomial since there is a relationship between bracket polynomial of virtual link diagram $D$ with signed Bollob\'{a}s-Riordan polynomial of its signed Tait graph $\textbf{G}_r$ in \cite{Chmu07,Chmu09}.
We leave the proof to the readers.
\end{remark}

\begin{remark}
It is obvious that for a signed planar cyclic graph $G$, $F[G](A,B,d)$ is equal to signed Tutte polynomial $Q[G](A,B,d)$ since $Q[G](A,B,d)=\langle G_D\rangle(A,B,d)$ in \cite{Kau4}.
This Tutte polynomial for signed abstract graphs is invariant under graph-theoretic analogs of the second and third Reidemeister moves (see \cite{Kau5}).
Thus it is an invariant of abstract graphs for these moves.
This suggests other directions of investigation that we shall take up in another paper.
We generalize the Tutte polynomial for signed graphs to a polynomial $F[G]$ for signed cyclic graphs in this paper.
\end{remark}

\section{Acknowledgements}

This work is supported by NSFC (No. 11671336) and President¡¯s Funds
of Xiamen University (No. 20720160011). Kauffman thanks the Simons Foundation for support under Cooperative Grant Number 426075.


\begin{thebibliography}{0}
\bibitem{Boll} B. Bollob\'{a}s, O. Riordan, A polynomial invariant of graphs on orientable surfaces, \textit{Proc. London Math. Soc.} \textbf{83}(3) (2001) 513-531.

\bibitem{Boll2} B. Bollob\'{a}s, O. Riordan, A polynomial invariant of graphs on surfaces, \textit{Math. Ann.} \textbf{323} (2002) 81-96.

\bibitem{Chmu07} S. Chmutov, I. Pak, The Kauffman bracket of virtual links and the Bollob¡äas-
Riordan polynomial, \textit{Mosc. Math. J.} \textbf{7}(3) (2007) 409¨C418.

\bibitem{Chmu09} S.Chmutov, Generalized duality for graphs on surfaces and the signed Bollob¡äas-
Riordan polynomial, \textit{J. Combin. Theory Ser. B} \textbf{99} (2009) 617¨C638.


\bibitem{Chmu17} S. Chmutov, Topological tutte polynomial, preprint (2017), arXiv: 1708.08132.

\bibitem{Mo} J. Ellis-Monaghan, I. Moffatt, \textit{Graphs on Surfaces: Dualities, Polynomials,
and Knots} (Springer, New York, 2013).

\bibitem{FIKM} R.Fenn, D. P. Ilyutko, L. H. Kauffman, V. O. Manturov, Unsolved problems in virtual knot theory and combinatorial knot theory, Knots in Poland III-Part III, Banach Center Publ.,Polish Acad. Sci. Inst. Math., Warsaw, \textbf{103} (2014) 9-61.

\bibitem{Gr} J. Green, A table of virtual knots, www.math.toronto.edu/drorbn/Students/GreenJ/.

\bibitem{DKK} H. A. Dye, A. Kaestner, L. H. Kauffman, Khovanov homology, Lee homology and a Rasmussen invariant for virtual knots. {\it J. Knot Theory Ramifications} \textbf{26}(3) (2017) Article ID: 1741001, (57 pp).

\bibitem{Ima} T. Imabeppu, On Sawollek polynomials of checkerboard colorable virtual links, \textit{J. Knot Theory Ramifications} \textbf{25}(2) (2016) Article ID: 1650010, (19 pp).

\bibitem{Kam} N. Kamada, S. Kamada, Abstract link diagrams and virtual knots, {\it J. Knot Theory Ramifications} \textbf{9}(1) (2000) 93-106.

\bibitem{K1} N. Kamada, On the Jones polynomials of checkerboard colorable virtual knots, {\it Osaka J. Math.} \textbf{39}(2) (2002) 325--333.

\bibitem{K2} N. Kamada, Span of the Jones polynomial of an alternating virtual link, {\it Algebr. Geom.
Topol.} \textbf{4} (2004) 1083-1101.

\bibitem{Kau5} L. H. Kauffman, New invariants in the theory of knots, {\it American Math. Monthly} \textbf{95}(3) (Mar., 1988) 195-242.

\bibitem{Kau1} L. H. Kauffman, Virtual knot theory. {\it European J. Combin.} \textbf{20}(7) (1999) 663-690.

\bibitem{Kau4} L. H. Kauffman, A Tutte polynomial for signed graphs, {\it Disc. Appl. Math.} \textbf{25} (1999) 105-127.

\bibitem{Kau3} L. H. Kauffman, Introduction to virtual knot theory, {\it J. Knot Theory Ramifications} \textbf{21}(13) (2012) 1240007-1-1240007-37.

\bibitem{Kau2} L. H. Kauffman Talks on Combinatorial Knot Theory in Xiamen University, June 2016.

\bibitem{Tutte} W. T. Tutte, A contribution to the theory of chromatic polynomials, {\it Canad. J. Math.} \textbf{6} (1954) 80-91.

\end{thebibliography}
\end{document}